\documentclass[11pt]{amsart}

\usepackage{amsmath,amssymb,amsthm}
\usepackage{graphicx}
\usepackage[hidelinks]{hyperref}

\newtheorem{theorem}{Theorem}[section]
\newtheorem{proposition}[theorem]{Proposition}
\newtheorem{lemma}[theorem]{Lemma}
\newtheorem{corollary}[theorem]{Corollary}
\newtheorem{conjecture}[theorem]{Conjecture}
\theoremstyle{definition}
\newtheorem{definition}[theorem]{Definition}
\newtheorem{problem}[theorem]{Problem}
\theoremstyle{remark}
\newtheorem{remark}[theorem]{Remark}

\newcommand{\dnorm}[1]{\lVert #1\rVert}          
\newcommand{\Z}{\mathbb{Z}}
\newcommand{\Q}{\mathbb{Q}}
\newcommand{\N}{\mathbb{N}}

\newcommand{\eps}{\varepsilon}
\newcommand{\pT}{p_{T}}
\DeclareMathOperator{\round}{round}

\begin{document}

\title[Superlinear complexity of the $(3/2)^n$ steering word]{Superlinear
complexity of the $(3/2)^n$ steering word}

\author{Ralf Stephan}
\thanks{Institute for Globally Distributed Open Research and Education (IGDORE)
}

\begin{abstract}
Write $(3/2)^n = m_n + \eps_n$ with $m_n$ the nearest integer and
$\eps_n\in[-\tfrac12,\tfrac12)$, and let $T=(t_n)$, $t_n=2m_{n+1}-3m_n$, be the
resulting \emph{steering word}: the step-by-step record of the map
$x\mapsto\tfrac32 x$ on the orbit of 1, coded by nearest-integer rounding.
We prove that the subword complexity $\pT(k)$ of $T$ is superlinear,
$\pT(k)/k\to\infty$. A pigeonhole argument reduces this to exponential pair
repulsion of the orbit modulo one, which we establish by two applications of
the $p$-adic Subspace Theorem: once in three variables, which handles all gaps
between the two orbit points at once, and once in two variables (Ridout's
theorem) for the single degenerate family of constant gap. The argument is
completely formalized in Lean~4.
\end{abstract}
\maketitle

\section{Introduction}\label{sec:intro}

For a real number $x$ we write $\dnorm{x}=\min_{n\in\Z}|x-n|$ for its distance to
the nearest integer. The distribution of the sequence $\big(\dnorm{\alpha^n}\big)$
for a fixed real $\alpha>1$ is poorly understood; already for $\alpha=3/2$ it is
open whether $\dnorm{(3/2)^n}\to 0$ can fail to hold densely. Mahler
\cite{Mah57} proved that for rational $\alpha>1$ that is not an integer, and any
$0<\ell<1$, the fractional part of $\alpha^n$ exceeds $\ell^n$ for all but
finitely many $n$; Corvaja and Zannier \cite{CZ04} recast and extended this
through the Schmidt subspace theorem. A version of Mahler's theorem with a
rational multiplier, a special case of Ridout's theorem \cite{Rid57} and of the
Main Theorem of \cite{CZ04}, is one of the two Diophantine inputs below.

This paper concerns not the sizes $\dnorm{(3/2)^n}$ themselves but the
\emph{symbolic dynamics} they generate. Writing $(3/2)^n=m_n+\eps_n$ with
$m_n\in\Z$ and $\eps_n\in[-\tfrac12,\tfrac12)$, the increments
\[
   t_n \;:=\; 2m_{n+1}-3m_n \;\in\;\{-2,-1,0,1,2\}
\]
record, at each step, which integer translate the multiplication-by-$3/2$ map
selects. The word $T=(t_n)_{n\ge0}$ is the full itinerary of that autonomous
dynamics; we call it the \emph{steering word}. It is the round-to-nearest
analogue of the itineraries of the $\tfrac32$-dynamics studied through rational
base number systems \cite{AFS08}. Its \emph{subword complexity}
$\pT(k)$ is the number of distinct length-$k$ factors it contains. Our main
result is the following.

\begin{theorem}[Main theorem]\label{thm:intro-main}
The subword complexity of the steering word $T$ is superlinear:
for every $C$ there is a $K$ with $\pT(k)>C\,k$ for all $k\ge K$; equivalently
$\pT(k)/k\to\infty$.
\end{theorem}

The proof has the shape of a transfer of the Adamczewski--Bugeaud method
\cite{AB07} for complexity lower bounds of algebraic numbers to the present
dynamical word, and
it simplifies in one respect: because the $\tfrac32$-dynamics is
autonomous, the bookkeeping is position-uniform, and every repeated factor of
$T$ compresses into a single clean Diophantine inequality, with no
prefix-anchoring. Concretely the argument runs in three stages.

\emph{Stage~0} (Section~\ref{sec:stage0}) is an exact identity: a length-$k$
factor of $T$ occurring at two positions $a<c$ forces
$3^k(\eps_c-\eps_a)=2^k(\eps_{c+k}-\eps_{a+k})$ and the divisibility
$2^k\mid m_c-m_a$. This already gives two unconditional theorems --- $T$ is not
eventually periodic (Theorem~\ref{thm:aperiodic}), and
$\pT(k)\ge (41/24)\,k-3$ (Theorem~\ref{thm:t01}), strictly above the universal
Morse--Hedlund floor $k+1$.

\emph{Stage~1} (Section~\ref{sec:stage1}) reduces
Theorem~\ref{thm:intro-main} by pigeonhole to a single statement:
\begin{quote}
   (K)\quad for every $\theta<1$, the inequality
   $\dnorm{(3/2)^c-(3/2)^a}\le\theta^{\,c}$ has only finitely many solutions
   $2\le a<c$.
\end{quote}
This is \emph{exponential pair repulsion} for the orbit --- an improvement of the
trivial repulsion floor $\dnorm{(3/2)^c-(3/2)^a}\ge 2^{-c}$ from an exponential
base $2^{-1}$ to an arbitrary base $\theta<1$.

\emph{Stage~2} (Section~\ref{sec:stage2}) establishes (K). To a violator
$(a,c)$ we attach the integer point $\big(3^c,\,2^{c-a}3^a,\,2^cM\big)$, where
$M$ is the integer nearest to $(3/2)^c-(3/2)^a$, and apply the $p$-adic
Subspace Theorem in three variables with the primes $2$ and $3$. The $2$- and
$3$-adic smallness of the coordinates exactly pays for their size, and the
theorem confines all violators to finitely many planes, whatever the gap
$c-a$. On each plane elementary arithmetic finishes the job, except on planes
of constant gap. There the problem has only one parameter left, and it falls to
the two-variable case of the same theorem, Ridout's theorem \cite{Rid57}.

\medskip
The elementary consequences of Stage~0 and the Stage~1 reduction are
unconditional. Stage~2 rests on the Evertse--Schlickewei $S$-arithmetic
subspace theorem \cite{Sch91,BG06}, used in three variables and in two; in the
formalization the two-variable input is obtained as the simplest case of the
Main Theorem of Corvaja--Zannier \cite{CZ04}, which is itself derived there.
The subspace theorem is therefore the single external input of the whole
development.

\subsection{Relation to the complexity landscape}\label{sec:landscape}

The conclusion $\pT(k)/k\to\infty$ is precisely the shape of the
Adamczewski--Bugeaud theorem \cite{AB07} for the $b$-ary expansions of irrational
algebraic numbers, and the analogy extends to the gap that theorem leaves open:
there Borel's conjecture predicts normality, that is $p(k)=b^{\,k}$, so
superlinearity is the unconditional shadow of an expected exponential truth. The
same is the case here, and it can be made precise.

Superlinearity already places $T$ outside every classical low-complexity family.
Automatic sequences satisfy $p(k)=O(k)$ \cite{Cob72,AS03}, as do linearly
recurrent words \cite{DHS99}, Sturmian words and the codings of rotations and of
interval exchanges, and fixed points of primitive substitutions;
Theorem~\ref{thm:intro-main} excludes $T$ from all of them. For the general theory of the factor complexity function we refer to
\cite{CN10}, and to \cite{Fer99,Puz25} for surveys.

What superlinearity does \emph{not} decide is whether $T$ is morphic, and here a
rigidity theorem of Pansiot \cite{Pan84} enters: the complexity of a pure morphic
word is one of $\Theta(1)$, $\Theta(k)$, $\Theta(k\log\log k)$, $\Theta(k\log k)$
and $\Theta(k^{2})$ --- in particular it is polynomially bounded. Hilion and
Levitt \cite{HL25} have recently proved the analogous four-class theorem for
right-infinite words representing attracting fixed points of an automorphism of a
finitely generated free group. The following two statements turn this rigidity
into a dichotomy for the steering word; both are consequences of
Lemma~\ref{lem:contraction}, read backwards.

\begin{proposition}[Packing bound]\label{prop:packing}
Let $k\in\N$ and let $A\subset\N$ be a finite set of positions whose orbit points
are pairwise more than $(2/3)^k$ apart, that is $|\eps_c-\eps_a|>(2/3)^k$ for all
$a\ne c$ in $A$. Then $\pT(k)\ge |A|$. Equivalently
$\pT(k)\ge N\big((2/3)^k\big)$, where $N(\delta)$ is the largest cardinality of a
$\delta$-separated subset of $\{\eps_n\}$.
\end{proposition}

\begin{proof}[Proof]
By Lemma~\ref{lem:contraction}, two positions carrying the same length-$k$ factor
satisfy $|\eps_c-\eps_a|\le(2/3)^k$. Hence the length-$k$ factors at the
positions of $A$ are pairwise distinct.
\end{proof}

\begin{corollary}[Density forces exponential complexity]\label{cor:dense}
If $\big((3/2)^n\big)_{n\ge0}$ is dense modulo one, then
$\pT(k)\ge\lceil(3/2)^k\rceil-1$ for every $k$.
\end{corollary}

\begin{proof}[Proof]
Fix $k$ and a positive integer $M<(3/2)^k$, so that $(2/3)^k<1/M$. The $M$ points
$-\tfrac12+j/M$, $0\le j<M$, lie in $[-\tfrac12,\tfrac12)$ and are $1/M$-separated.
Approximating each of them within $\delta=\tfrac13\big(1/M-(2/3)^k\big)>0$ by a
point of the orbit produces $M$ positions whose $\eps$-values are pairwise more
than $1/M-2\delta=\tfrac13\cdot\tfrac1M+\tfrac23(2/3)^k>(2/3)^k$ apart, so
$\pT(k)\ge M$ by Proposition~\ref{prop:packing}. Since
$\lceil(3/2)^k\rceil-1<(3/2)^k$, the claim follows.
\end{proof}

Two readings of Corollary~\ref{cor:dense} place the main theorem. First, it
identifies Theorem~\ref{thm:intro-main} as the weakest nontrivial instance of what
the density conjecture predicts: the expected truth is exponential, and the
difficulty of the subject lies in the gap between $k\cdot\omega(1)$ and
$(3/2)^k$. Second, read contrapositively, \emph{any} subexponential upper bound
$\pT(k)=o\big((3/2)^k\big)$ would refute density of $\big((3/2)^n\big)$ modulo
one. This is why no upper bound of that strength is offered here: beyond the
trivial alphabet ceiling $\pT(k)\le5^k$, the sharp known environment is Kopra's
trace subshift \cite{Kop21}, which gives $\pT(k)\le 4\cdot3^k-3\cdot2^k$ --- of
base $3$ rather than $3/2$.

Combining Corollary~\ref{cor:dense} with Pansiot's theorem: if
$\big((3/2)^n\big)$ is dense modulo one, then $T$ is not morphic, and by
\cite{HL25} it does not represent an attracting fixed point of a free group
automorphism either; equivalently, exhibiting a morphism that generates $T$ would
refute the density conjecture. It is Theorem~\ref{thm:intro-main} that makes the
dichotomy non-vacuous, by eliminating the classes $\Theta(1)$ and $\Theta(k)$
outright.

The implication runs one way only. Proposition~\ref{prop:packing} bounds the
complexity below by a packing number, not conversely: the letter $t_n$ is a
function of $(\eps_n,m_n\bmod 2)$, but $m_{n+1}\bmod 2$ already requires
$m_n\bmod 4$, so a length-$k$ window is governed by $(\eps_a,m_a\bmod 2^k)$ and
distinct factors need not arise from separated orbit points. We claim no converse
to Lemma~\ref{lem:contraction}, and Theorem~\ref{thm:intro-main} accordingly
transfers nothing back to the density question. The complementary unconditional
information is due to Flatto, Lagarias and Pollington \cite{FLP95}, whose theorem
gives the set of limit points of $\{\xi(3/2)^n\}$ a diameter of at least $1/3$:
the orbit closure is known to be spread out, but is not known to have positive box
dimension --- which by Proposition~\ref{prop:packing} would already force
$\pT(k)\ge c^{\,k}$ for some $c>1$.

Finally, on the shape of $\pT$ as a function of $k$. The region bracketed by the
bounds above is genuinely populated: Cassaigne \cite{Cas03} constructs words whose
complexity grows faster than every polynomial and slower than every exponential;
and in the linear regime the pairs $(\alpha,\beta)$ realizable as
$\alpha=\liminf_k \pT(k)/k$ and $\beta=\limsup_k \pT(k)/k$ constitute the Heinis
spectrum, recently shown by Erazo and Moreira \cite{EM25} to have non-empty
interior.

\subsection*{Setup and notation}
Throughout, $\dnorm{\cdot}$ is the distance to the nearest integer, computed
exactly in $\Q$ for all the rationals that occur.

\begin{definition}[The orbit decomposition and the steering word]\label{def:objects}
For $n\in\N$ put
\[
   m_n:=\round\big((3/2)^n\big),\qquad
   \eps_n:=(3/2)^n-m_n\in[-\tfrac12,\tfrac12),\qquad
   R_n:=3^n-2^n m_n,
\]
so that $\eps_n=R_n/2^n$; here $\round$ rounds half-integers up, placing
$\eps_n$ in the half-open window $[-\tfrac12,\tfrac12)$. The \emph{steering
letter} is $t_n:=2m_{n+1}-3m_n$, the \emph{steering word} is $T=(t_n)_{n\ge0}$,
and the \emph{parity word} is $b=(b_n)$, $b_n:=m_n\bmod 2$.
\end{definition}

Elementary manipulations show
$t_n=3\eps_n-2\eps_{n+1}$, whence $|t_n|\le 2$ (a five-letter alphabet) and
$t_n\equiv m_n\pmod 2$, so that $b$ is the letter-to-letter reduction of $T$
modulo $2$. Moreover $R_n$ is odd for $n\ge1$; this parity is the source of the
trivial repulsion floor. The sequence begins
$m=1,2,2,3,5,8,\dots$ and $t=1,-2,0,1,1,\dots$. All objects live in $\Z$ or $\Q$;
no real analysis enters before Section~\ref{sec:stage2}.

\begin{definition}[Circuit sum]\label{def:circuit}
For $a,k\in\N$ put
\[
   W(a,k)\;:=\;\sum_{i<k} 3^{\,k-1-i}\,2^{\,i}\,t_{a+i}\ \in\ \Z .
\]
\end{definition}

\section{The repetition identity}\label{sec:stage0}

The whole analysis rests on one exact cancellation. Iterating the step relation
$2m_{n+1}=3m_n+t_n$ over a window telescopes the circuit sum into a closed form.

\begin{lemma}[Closed form of the circuit sum]\label{lem:closed}
For all $a,k\in\N$,
\[
   W(a,k)\;=\;3^k\eps_a-2^k\eps_{a+k},
   \qquad\text{equivalently}\qquad
   2^k m_{a+k}=3^k m_a+W(a,k).
\]
\end{lemma}

\begin{proof}[Proof]
Induction on $k$, using the recurrence $W(a,k+1)=3\,W(a,k)+2^k t_{a+k}$ and
$t_n=3\eps_n-2\eps_{n+1}$; the main terms of the orbit cancel exactly.
\end{proof}

\begin{definition}[Repetition]\label{def:rep}
A \emph{repetition} of length $k$ at positions $a<c$ is an equality of factors,
$t_{a+i}=t_{c+i}$ for all $i<k$. (Occurrences may overlap; nothing anchors them to
a prefix.)
\end{definition}

Equal factors accumulate equal circuit sums, so subtracting two instances of the
closed form gives the central identity.

\begin{theorem}[The repetition identity]\label{thm:lemmaR}
If there is a repetition of length $k$ at $a<c$, then
\[
   3^k(\eps_c-\eps_a)=2^k(\eps_{c+k}-\eps_{a+k}),
   \qquad
   3^k(m_c-m_a)=2^k(m_{c+k}-m_{a+k}).
\]
\end{theorem}

\begin{proof}[Proof]
By Definition~\ref{def:rep} the circuit sums agree, $W(a,k)=W(c,k)$. Substituting
the two forms of Lemma~\ref{lem:closed} and cancelling yields both identities
(the fractional form from the $\eps$-closed form, the integer form from its
integer companion).
\end{proof}

\begin{corollary}[Divisibilities]\label{cor:dvd}
Under the hypothesis of Theorem~\ref{thm:lemmaR},
\[
2^k\mid m_c-m_a\qquad \text{and} \qquad3^k\mid m_{c+k}-m_{a+k}.
\]
\end{corollary}

\begin{proof}[Proof]
$2^k$ divides $3^k(m_c-m_a)$ by the integer identity, and $\gcd(2^k,3^k)=1$;
symmetrically for $3^k$.
\end{proof}

The divisibility upgrades to size once $m$ is strictly increasing, which it is
from index $2$ on. This is the mechanism that forbids long repetitions late in
the word.

\begin{proposition}[Growth ceiling]\label{prop:ceiling}
If there is a repetition of length $k$ at $2\le a<c$, then
\[
   2^{\,k+c+1}\le 3^{\,c+1},
\]
that is, $k\le (c+1)\log_2(3/2)\approx 0.585\,(c+1)$.
\end{proposition}

\begin{proof}[Proof]
Corollary~\ref{cor:dvd} gives $2^k\mid m_c-m_a$, and $m_a<m_c$ (strict
monotonicity of $m$ for indices $\ge2$), so $2^k\le m_c-m_a\le m_c$. Combined
with the a~priori bound $2\,(2^c m_c)\le 2\cdot 3^c+2^c$ --- itself a restatement
of $\eps_c\ge-\tfrac12$ --- this yields $2^{k+c+1}\le 3^{c+1}$ as an inequality
of integers.
\end{proof}

\begin{remark}
The same computation bounds a $(j{+}1)$-fold repetition of a length-$p$ block
starting at $a\ge2$ by $2^{\,jp+(a+p)+1}\le 3^{\,a+p+1}$: long periodic windows
cannot occur late in the word.
\end{remark}

The identity has two further quantitative shadows, the interface to the
Diophantine kernel of Section~\ref{sec:stage1}.

\begin{lemma}[Contraction]\label{lem:contraction}
A repetition of length $k$ at $a<c$ forces $|\eps_c-\eps_a|\le (2/3)^k$.
\end{lemma}

\begin{proof}[Proof]
Rewrite the fractional identity of Theorem~\ref{thm:lemmaR} as
\[
\eps_c-\eps_a=(2/3)^k(\eps_{c+k}-\eps_{a+k})
\]
and bound the difference of two
centered fractional parts by $1$.
\end{proof}

\begin{lemma}[Trivial repulsion floor]\label{lem:floor}
For $1\le a<c$,
\[
2^c\,\dnorm{(3/2)^c-(3/2)^a}\ge 1.
\]
\end{lemma}

\begin{proof}[Proof]
The orbit difference is
\[
3^a(3^{\,c-a}-2^{\,c-a})/2^c,
\]
a fraction with odd
numerator over $2^c$; hence its distance to the nearest integer is at least
$2^{-c}$. (Equivalently $2^c(\eps_c-\eps_a)$ is an odd integer.)
\end{proof}

Lemma~\ref{lem:floor} is the estimate the kernel (K) will improve: from the
exponential base $2^{-1}$ to an arbitrary base $\theta<1$. Two unconditional
theorems already drop out of Stage~0.

\begin{theorem}[Aperiodicity]\label{thm:aperiodic}
The steering word $T$ is not eventually periodic.
\end{theorem}

\begin{proof}[Proof]
If $t_{n+p}=t_n$ for all $n\ge N$ with $p\ge1$, then at $a=\max(N,2)$ and
$c=a+p$ there is a repetition of \emph{every} length $k$. Taking $k=3^{\,c+1}$
contradicts the growth ceiling $2^{\,k+c+1}\le 3^{\,c+1}$ of
Proposition~\ref{prop:ceiling}, since $3^{\,c+1}<2^{\,3^{c+1}}$. (Aperiodicity of
this family is also a special case of \cite[Lemma~1]{DN05}.)
\end{proof}

\begin{theorem}[Complexity floor]\label{thm:t01}
For all $k$,
\[
   41\,k\;\le\;24\,\pT(k)+72,
   \qquad\text{i.e.}\qquad
   \pT(k)\ \ge\ \tfrac{41}{24}\,k-3\ \approx\ 1.708\,k.
\]
\end{theorem}

\begin{proof}[Proof]
By Proposition~\ref{prop:ceiling} and the certified estimate $3^{41}\le 2^{65}$
(which encodes $\log_2 3\le 65/41$), a repetition of length $k$ at $2\le a<c$
forces $41k\le 24c+24$. Hence the windows starting at
$2,3,\dots,C$ with $C=\lfloor(41k-25)/24\rfloor$ are pairwise distinct as
factors, giving at least $C-1$ distinct length-$k$ factors; comparing with the
total count $\pT(k)$ and simplifying yields the stated inequality.
\end{proof}

\begin{remark}
The constant intrinsic to the method is $1/\log_2(3/2)\approx 1.7095$; the
rational slope $41/24\approx1.7083$ formalized here is the certified rounding of
it obtained from the integer certificate $3^{41}\le 2^{65}$. In either form the
bound lies strictly above the Morse--Hedlund floor $k+1$ valid for every
aperiodic word \cite{MH38}. A lower bound of exactly this shape, with the
constant $1/\log_2(3/2)$, is due to Dubickas \cite{Dub09} for the digit words of
$\lceil(p/q)^n\rceil$; the present statement is its transfer to the
nearest-integer steering word of the orbit of $1$.
\end{remark}

\subsection*{The parity word}
The primary target is the five-letter word $T$, not its binary reduction $b$.
Since $b$ is a letter-to-letter image of $T$ we have $\pT(k)\ge p_b(k)$, so
Theorem~\ref{thm:t01} says nothing about $p_b$. What survives is a dichotomy.

\begin{proposition}[Parity dichotomy]\label{prop:parity}
A length-$k$ repetition of the parity word $b$ at positions $a,c$ either lifts to
a repetition of $T$ (so Theorem~\ref{thm:t01} applies), or the two windows first
disagree at some step $j<k$ at which the steering letters differ by a nonzero even
amount.
\end{proposition}

\begin{proof}[Proof]
If the $T$-windows agree, we are in the first case. Otherwise let $j$ be the
least index of disagreement; there $b_{a+j}=b_{c+j}$ forces
$t_{c+j}-t_{a+j}$ even, and it is nonzero by choice of $j$.
\end{proof}

\begin{remark}
While the two orbit windows remain $T$-synchronized their $\eps$-difference
expands by exactly $3/2$ per step,
\[
\eps_{c+j}-\eps_{a+j}=(3/2)^j(\eps_c-\eps_a).
\]
Consequently a \emph{maximal} desynchronization $|t_{c+j}-t_{a+j}|=4$ after $j$
synchronized steps certifies $|\eps_c-\eps_a|\ge (2/3)^{j+1}$, whereas a minimal
one $|t_{c+j}-t_{a+j}|=2$ carries no such constraint. This is the precise sense in
which superlinearity for $b$ is harder; we claim no unconditional linear bound for
$p_b$.
\end{remark}

\section{Reduction to a Diophantine kernel}\label{sec:stage1}

\begin{definition}[The kernel]\label{def:kernel}
For a rational scale $\theta\in(0,1)$ let
\[
   V(\theta)\;:=\;\big\{(a,c):\,2\le a<c,\ \dnorm{(3/2)^c-(3/2)^a}\le\theta^{\,c}\big\}.
\]
We say \emph{pair repulsion} holds at $\theta$ if $V(\theta)$ is finite, and we
call the assertion
\[
   \text{(K)}\qquad V(\theta)\ \text{is finite for every rational}\ \theta\in(0,1)
\]
the \emph{Diophantine kernel}. The target property is that $\pT$ be
\emph{superlinear}: for every $C$ there is a $K$ with $\pT(k)>C\,k$ for all
$k\ge K$.
\end{definition}

Restricting $\theta$ to the rationals loses nothing: $V(\theta)$ is monotone in
$\theta$ and $\Q$ is dense, so the rational and real formulations of (K) are
equivalent. This keeps the reduction inside $\Q$.

\begin{theorem}[Reduction]\label{thm:reduction}
The kernel {\rm(K)} implies that $\pT$ is superlinear.
\end{theorem}

\begin{proof}[Proof]
Fix $C$ and suppose $\pT(k)\le C\,k$ for some large $k$. Among the $Ck+1$ windows
starting at positions $2,\dots,Ck+2$, pigeonhole produces two equal factors,
i.e.\ a repetition of length $k$ at some $2\le a<c\le (C+2)k$. Choose once and for
all a rational scale $\theta<1$ with $\theta^{\,C+2}\ge 2/3$ (a Bernoulli
estimate provides one). By Lemma~\ref{lem:contraction},
\[
   \dnorm{(3/2)^c-(3/2)^a}\le |\eps_c-\eps_a|\le (2/3)^k\le \theta^{\,(C+2)k}\le\theta^{\,c},
\]
so $(a,c)\in V(\theta)$. As $k$ grows, the growth ceiling
(Proposition~\ref{prop:ceiling}, in the form $41k\le 24c+24$) forces $c\to\infty$,
contradicting the finiteness of $V(\theta)$ granted by (K). Hence no such $C$
bounds $\pT(k)/k$.
\end{proof}

The kernel is thus the sole remaining obstacle, and it is unconditional as a
\emph{reduction}: (K) enters Theorem~\ref{thm:reduction} only as a hypothesis.
The rest of the paper establishes it.

\section{The kernel via the Subspace Theorem}\label{sec:stage2}

The kernel (K) is a statement in two parameters, the position $a$ and the gap
$s=c-a$. We prove it by applying the Subspace Theorem once in three variables,
which confines all violators to finitely many planes at once, whatever their
gaps. Elementary arithmetic disposes of every plane except those on which the
gap is constant. On such a plane only one parameter is left, and the
two-variable case of the same theorem (Ridout's theorem) finishes the proof.

\subsection{The Subspace Theorem}

It helps to say informally what the theorem does before the formulas begin.
Its ancestor is Roth's theorem: an algebraic irrational admits no rational
approximations appreciably better than the trivial ones. Schmidt's subspace
theorem is the higher-dimensional form of the same phenomenon. If several
linear forms take simultaneously much smaller values at an integer point than
the size of that point warrants, then the point is not free to lie anywhere:
all such points fall into finitely many proper subspaces, that is, they satisfy
one of finitely many fixed linear relations with fixed coefficients. Unusually
good approximation remains possible, but only along degenerate families.

This bites here because the quantities we must control, $(3/2)^a$ and
$(3/2)^c$, are built from the two primes $2$ and $3$ alone. They are huge
numbers, of height growing exponentially in the exponent, and yet they are
described by very few integer parameters. To say that a small integer
combination of such numbers lies unexpectedly close to an integer is to say
that a linear form is unexpectedly small at a large integer point. The
$S$-arithmetic version of the theorem, due to Schlickewei \cite{Sch77}, measures
smallness at the primes $2$ and $3$ as well as at the real place. This is what
makes the sparse multiplicative structure visible to the machinery.

For a prime $p$ let $|x|_p=p^{-v_p(x)}$ denote the $p$-adic absolute value on
$\Q$, and write $|x|_\infty=|x|$. A nonzero integer $n$ satisfies the product
formula $|n|\prod_p|n|_p=1$, and hence
\begin{equation}\label{eq:accounting}
   |n|\cdot|n|_2\cdot|n|_3\;\ge\;1,
   \qquad\text{with equality if and only if } n=\pm2^i3^j .
\end{equation}
In words: measured at all three places together, an integer of the form
$\pm2^i3^j$ (an \emph{$S$-unit}, for $S=\{2,3\}$) has size exactly $1$, however large it is at the real place, and a
general integer $n=2^i3^jn'$ with $\gcd(n',6)=1$ has size $|n'|$. This
accounting principle is behind both applications below.

\begin{theorem}[Subspace Theorem, $p$-adic form {\cite{Sch77}; see \cite[Ch.~7]{BG06}}]\label{thm:sub}
Let $n\ge2$, let $S$ be a finite set of primes, and for each
$v\in S\cup\{\infty\}$ let $L_{v,1},\dots,L_{v,n}$ be linearly independent linear
forms in $n$ variables with rational coefficients. For every $\epsilon>0$ the
solutions $x\in\Z^{n}\setminus\{0\}$ of
\[
   \prod_{v\in S\cup\{\infty\}}\ \prod_{i=1}^{n}\,|L_{v,i}(x)|_v
   \;\le\;\|x\|^{-\epsilon},
   \qquad \|x\|:=\max_i|x_i|,
\]
lie in finitely many proper linear subspaces of $\Q^{n}$.
\end{theorem}

Throughout, $S=\{2,3\}$, and the forms at $2$ and at $3$ are the coordinate
forms $X_1,\dots,X_n$; only the forms at $\infty$ vary. With this choice, by
\eqref{eq:accounting}, a coordinate of $x$ that is of the form $\pm2^i3^j$
contributes nothing to the double product as long as the coordinate form
$X_i$ is also one of the forms at $\infty$. The product can therefore only be
small because of the remaining forms at $\infty$, which measure the
near-coincidence one wants to rule out.

Two features of Theorem~\ref{thm:sub} shape everything that follows. First, it
does not say that there are only finitely many solutions, and it could not: a
one-parameter family of solutions can lie in a single subspace. Each of the
finitely many subspaces must therefore be examined by hand. Second, it is
ineffective: it neither bounds the solutions nor identifies the subspaces. All
results of this section inherit this ineffectivity.

\subsection{Constant gap: two variables}

\begin{lemma}[Mahler's theorem with a rational multiplier]\label{lem:ridout}
Let $\xi\in\Q_{>0}$ and $0<\theta<1$. Then only finitely many $a\in\N$ satisfy
\[
   0\;<\;\dnorm{\xi(3/2)^{a}}\;\le\;\theta^{a}.
\]
\end{lemma}

The meaning of the lemma is best seen by dividing through by $(3/2)^a$. If $M$
is the integer nearest to $\xi(3/2)^a$, the inequality says
$|\xi-M2^a/3^a|\le\theta^a(2/3)^a$. This is an approximation to the fixed
rational $\xi$ by fractions whose denominator is a power of $3$ and whose
numerator is divisible by a large power of $2$. Ridout's theorem
\cite{Rid57} says that the $p$-adic constraints on numerator and denominator
are paid for in exactly the way \eqref{eq:accounting} describes, so that such
approximations cannot be appreciably better than trivial. The proof makes this
precise.

\begin{proof}[Proof]
Write $\xi=u/v$ with $u,v\in\N$, and let $M$ be the integer nearest to
$\xi(3/2)^a$; for $a$ large, $M\ge1$. Put
\[
   y=(y_1,y_2):=\big(u3^{a},\,vM2^{a}\big)\in\Z^{2},
   \qquad
   N':=y_1-y_2=v2^{a}\big(\xi(3/2)^a-M\big),
\]
so that $0<|N'|\le v\,2^{a}\theta^{a}$. Since $|N'|\le u3^a$ for $a$ large, we
have $\|y\|\le u3^a+|N'|\le 2u\,3^{a}$. Take the forms $Y_1-Y_2$ and $Y_2$ at
$\infty$. Since $|y_1|_2\le1$, $|y_2|_2\le2^{-a}$, $|y_1|_3\le3^{-a}$,
$|y_2|_3\le1$ and $M\le\xi(3/2)^a+1$,
\[
   |N'|\cdot|y_2|\cdot|y_1|_2|y_2|_2\cdot|y_1|_3|y_2|_3
   \;\le\;|N'|\cdot vM2^a\cdot2^{-a}\cdot3^{-a}
   \;\le\;v^{2}(\xi+1)\,\theta^{a}.
\]
With $\epsilon:=\log(1/\theta)/(2\log3)$ we have
$\|y\|^{-\epsilon}\ge(2u)^{-\epsilon}\theta^{a/2}$, so the right-hand side is at
most $\|y\|^{-\epsilon}$ for $a\ge a_0$. By Theorem~\ref{thm:sub} with $n=2$,
the points $y$ lie on finitely many lines through the origin; none of them is
the line $Y_1=0$, as $y_1\ne0$, so each has the form $Y_2=\rho Y_1$. On such a
line $N'=(1-\rho)u3^{a}$. The value $\rho=1$ is excluded by $N'\ne0$, and for
$\rho\ne1$ the inequality $|1-\rho|\,u\,3^{a}\le v2^{a}\theta^a$ bounds $a$.
\end{proof}

For $\xi=1$ Lemma~\ref{lem:ridout} is Mahler's theorem \cite{Mah57}. It is also
the simplest instance (group $\langle2,3\rangle$, degree $d=1$, integer
multiplier $q=1$) of the Main Theorem of Corvaja and Zannier \cite{CZ04}, and
it is in that form, derived from Theorem~\ref{thm:sub} with $n=2$, that it
enters the formalization. We need it only for one family of multipliers.

\begin{corollary}[Gap-bounded slice]\label{cor:gapbounded}
For every $S\in\N$ and every rational $\theta\in(0,1)$, the set of
$(a,c)\in V(\theta)$ with $c\le a+S$ is finite.
\end{corollary}

\begin{proof}[Proof]
Fix a gap $1\le s_0\le S$. For $c=a+s_0$,
\[
   (3/2)^{c}-(3/2)^{a}=\xi\,(3/2)^{a},\qquad \xi:=(3/2)^{s_0}-1\in\Q_{>0}.
\]
The distance of this number to $\Z$ is positive by Lemma~\ref{lem:floor}, and
it is at most $\theta^{c}\le\theta^{a}$. Lemma~\ref{lem:ridout} leaves finitely
many $a$. Taking the union over the finitely many gaps proves the claim.
\end{proof}

\begin{remark}
The same input shows that fixed-gap repetitions are short in a strong sense:
for each $s_0\ge1$ and each rational slope $\mu>0$, only finitely many pairs
$(a,k)$ satisfy $k\ge\mu\,a$ together with a repetition of length $k$ at
$(a,a+s_0)$.
\end{remark}

\subsection{Exponential pair repulsion: three variables}

\begin{theorem}[Exponential pair repulsion; the kernel (K) holds]\label{thm:kernel}
For every rational $\theta\in(0,1)$ the set $V(\theta)$ is finite: only finitely
many pairs $2\le a<c$ satisfy
\[
   \dnorm{(3/2)^{c}-(3/2)^{a}}\;\le\;\theta^{c}.
\]
\end{theorem}

The idea is to treat the unknown integer nearest to $(3/2)^c-(3/2)^a$ as a
third coordinate. The near-integer condition then becomes the smallness of a
linear form at an integer point whose other coordinates are built from $2$ and
$3$.

\begin{proof}[Proof]
Let $(a,c)\in V(\theta)$, put $s:=c-a\ge1$, and let $M$ be the integer nearest
to $(3/2)^{c}-(3/2)^{a}$. Since $a\ge2$, this number is at least
$(3/2)^a\cdot\frac12\ge\frac98$, so $M\ge1$. Put
\[
   x=(x_1,x_2,x_3):=\big(3^{c},\ 2^{s}3^{a},\ 2^{c}M\big)\in\Z^{3},\qquad
   N:=x_1-x_2-x_3=2^{c}\big((3/2)^c-(3/2)^a-M\big).
\]
We collect four facts about $x$.
\begin{itemize}
\item $N$ is odd, since $x_1=3^c$ is odd and $x_2,x_3$ are even; in particular
   $N\ne0$. (This is Lemma~\ref{lem:floor}.)
\item $|N|=2^c\dnorm{(3/2)^c-(3/2)^a}\le2^{c}\theta^{c}$.
\item $M\le(3/2)^c-(3/2)^a+\frac12\le(3/2)^{c}$.
\item $\|x\|=x_1=3^{c}$. Indeed $x_2=2^s3^a<3^s3^a=3^c$; and
   $x_3=2^cM\le 3^c-2^s3^a+2^{c-1}<3^c$, because
   $2^s3^a>2^s2^{a-1}=2^{c-1}$.
\end{itemize}
Take the forms $X_1-X_2-X_3$, $X_2$, $X_3$ at $\infty$; they are linearly
independent. Using $|x_1|_2=1$, $|x_2|_2=2^{-s}$, $|x_3|_2\le2^{-c}$,
$|x_1|_3=3^{-c}$, $|x_2|_3=3^{-a}$ and $|x_3|_3\le1$,
\[
   \prod_{v}\prod_{i}|L_{v,i}(x)|_v
   \;\le\;\underbrace{|N|\cdot2^{s}3^{a}\cdot2^{c}M}_{v=\infty}
   \cdot\underbrace{2^{-s}2^{-c}}_{v=2}\cdot\underbrace{3^{-c}3^{-a}}_{v=3}
   \;=\;|N|\,M\,3^{-c}\;\le\;\theta^{c}\;=\;\|x\|^{-\epsilon},
   \qquad\epsilon:=\frac{\log(1/\theta)}{\log3}.
\]
By Theorem~\ref{thm:sub} with $n=3$, the points $x$ lie in finitely many proper
subspaces $\alpha X_1+\beta X_2+\gamma X_3=0$ of $\Q^{3}$. Since $x$ determines
$(a,c)$ (from $x_1$ we read off $c$, and then $a$ from $x_2$), it suffices to
show that each of these subspaces contains only finitely many of the points.
Fix one.

\emph{Case $\gamma=0$.} Then $\alpha3^{c}=-\beta2^{s}3^{a}$ with
$(\alpha,\beta)\ne(0,0)$, so both are nonzero and $(3/2)^{s}=-\beta/\alpha$.
The gap $s=s_0$ is therefore constant on this subspace, and
Corollary~\ref{cor:gapbounded} applies.

\emph{Case $\gamma\ne0$.} Normalise $\gamma=-1$. Then
$2^{c}M=\alpha3^{c}+\beta2^{s}3^{a}$, that is,
$M=\alpha(3/2)^{c}+\beta(3/2)^{a}$. The nearest integer is thus a fixed
rational combination of the two orbit points, and therefore so is the distance:
\[
   \dnorm{(3/2)^{c}-(3/2)^{a}}
   \;=\;\big|(3/2)^{c}-(3/2)^{a}-M\big|
   \;=\;\big|\beta_1(3/2)^{c}+\beta_2(3/2)^{a}\big|,
   \qquad \beta_1:=1-\alpha,\ \beta_2:=-1-\beta .
\]
This distance is at most $\theta^{c}$, and it is positive because $N\ne0$.
\begin{itemize}
\item If $\beta_1=\beta_2=0$, the distance vanishes, which is impossible.
\item If $\beta_1=0\ne\beta_2$, the distance is $|\beta_2|(3/2)^{a}\ge|\beta_2|$,
   so $\theta^{c}\ge|\beta_2|$, which bounds $c$.
\item If $\beta_1\ne0$, choose $s_1$ with $|\beta_1|(3/2)^{s_1}\ge2|\beta_2|$.
   For gaps $s\ge s_1$,
   \[
      \big|\beta_1(3/2)^{c}+\beta_2(3/2)^{a}\big|
      \;\ge\;(3/2)^{a}\big(|\beta_1|(3/2)^{s}-|\beta_2|\big)
      \;\ge\;\tfrac12|\beta_1|(3/2)^{c}\;\ge\;\tfrac12|\beta_1|,
   \]
   so $\theta^{c}\ge\tfrac12|\beta_1|$, which bounds $c$. The pairs with gap
   $s<s_1$ lie in a gap-bounded slice, finite by
   Corollary~\ref{cor:gapbounded}.
\end{itemize}
A bound on $c$ leaves finitely many pairs $a<c$. So in every case the subspace
contains only finitely many of the points $x$.
\end{proof}

\begin{remark}[What the proof does]\label{rem:why}
Two steps in the proof carry the weight, and it may help to see why each works.

\emph{The collapse of the double product.} Of the nine factors in the product,
those belonging to the $S$-unit coordinates $x_1=3^c$ and $x_2=2^s3^a$ cancel
exactly, by \eqref{eq:accounting}. The coordinate $x_3=2^cM$ is an $S$-unit
times the unknown integer $M$; its $2$-adic smallness pays for the factor
$2^c$, and only $M\le(3/2)^c$ survives. What is left is $|N|\,M\,3^{-c}$, with
$N$ the near-miss itself. Without the hypothesis on $(a,c)$ one only knows
$|N|\le2^{c-1}$, and the product is then at most $\frac12$, which is no saving
at all. The hypothesis $\dnorm{(3/2)^c-(3/2)^a}\le\theta^c$ brings it down to
$\theta^c=\|x\|^{-\epsilon}$, a fixed power of the size of the point, and a
power saving is exactly what Theorem~\ref{thm:sub} needs. The gap $s$ enters only through
$x_2=2^s3^a$, whose archimedean factor $2^s$ is repaid by its $2$-adic factor
$2^{-s}$. This is why a single application covers all gaps at once, and why no case distinction in the gap is needed until the
subspaces are examined.

\emph{The exceptional subspaces.} The constant-gap points
$(3^c,\,2^{s_0}3^{c-s_0},\,2^cM)$ all lie on the plane
$X_2=(2/3)^{s_0}X_1$, whatever $M$ is. Theorem~\ref{thm:sub} in three variables
therefore cannot rule them out: they form a one-parameter family inside one of
its subspaces. Inside that plane only a two-variable problem remains, and it is
solved by Lemma~\ref{lem:ridout}. Every other subspace is harmless because the
ratio of the two orbit points is $(3/2)^{c}/(3/2)^{a}=(3/2)^{s}$, which tends
to infinity with the gap. A fixed combination
$\beta_1(3/2)^c+\beta_2(3/2)^a$ can be small only if this ratio is close to
$-\beta_2/\beta_1$. Since the ratio cannot accumulate, a small combination
forces the gap into a finite set, and there is nothing further to control.
\end{remark}

\begin{theorem}[Main theorem]\label{thm:M4}
The subword complexity of the steering word is superlinear, $\pT(k)/k\to\infty$.
\end{theorem}

\begin{proof}[Proof]
Theorem~\ref{thm:kernel} supplies the kernel (K), and the reduction
(Theorem~\ref{thm:reduction}) turns it into superlinearity.
\end{proof}

\begin{remark}[Dependencies]\label{rem:status}
The results of Sections~\ref{sec:stage0} and \ref{sec:stage1} are
unconditional. The results of this section, and hence the main theorem, depend
on exactly one external input beyond the ambient logical framework: the
Evertse--Schlickewei $S$-arithmetic subspace theorem (\cite{Sch91}; see
\cite[Ch.~7]{BG06} for the form used), recorded once in the formal development.
It is used twice, with $n=3$ in Theorem~\ref{thm:kernel} and with $n=2$ in
Lemma~\ref{lem:ridout} (through the Main Theorem of Corvaja--Zannier, which is
derived from it inside the formalization). Every dependency is ineffective; no
bound on the exceptional sets is obtained.
\end{remark}

\section{Implications and further directions}\label{sec:outlook}

The results above sit at the junction of three programmes and leave a different
residue in each.

\subsection{Diophantine approximation}

Detached from its dynamical origin, the kernel of Theorem~\ref{thm:kernel}
reads
\[
   \#\big\{(a,c):2\le a<c,\ \dnorm{(3/2)^{c}-(3/2)^{a}}\le\theta^{\,c}\big\}
   \;<\;\infty
   \qquad\text{for every }\theta<1 .
\]
Mahler's theorem, in the form used by Zudilin \cite{Zud07}, says that
$\dnorm{(3/2)^{k}}\le\theta^{\,k}$ has finitely many solutions for every
$\theta<1$; the kernel is exactly its two-parameter companion. The orbit of $1$
under $x\mapsto\tfrac32x$ does not merely repel the integers: its points repel
\emph{each other} modulo one at every exponential rate. Only the base
$\theta=\tfrac12$ of that statement is elementary (Lemma~\ref{lem:floor});
everything above it is Subspace.

\emph{Effectivity.} Every result of Section~\ref{sec:stage2} is
ineffective, so no explicit $K$ can be extracted in
Theorem~\ref{thm:intro-main}. For the one-parameter problem the effective
theory is in better shape, though it barely clears the trivial base: after Baker
and Coates, and Beukers, Zudilin \cite{Zud07} proved
$\dnorm{(3/2)^{k}}>0.5803^{\,k}$ for all $k$ beyond an effectively computable
bound, against the trivial $2^{-k}$. Nothing comparable is known for pairs. The
reduction of Theorem~\ref{thm:reduction} converts any effective slice of the
kernel directly into an effective complexity bound, and the exchange rate is
explicit.

\begin{problem}\label{prob:eff}
Exhibit a $\theta\in(0,1)$ and an effectively computable $A(\theta)$ such that
every pair $(a,c)\in V(\theta)$ satisfies $c\le A(\theta)$.
\end{problem}

Such a $\theta$ yields, by the proof of Theorem~\ref{thm:reduction}, the
effective bound $\pT(k)>C\,k$ for all explicitly large $k$, with
$C=\log(2/3)/\log\theta-2$. The threshold to beat is therefore
$\theta=(2/3)^{24/89}\approx0.896$: below it the certified slope $41/24$ of
Theorem~\ref{thm:t01} is already better, and at $\theta$ near Zudilin's
$0.5803$ the exchange rate returns nothing at all. The whole difficulty of
Problem~\ref{prob:eff} lies in that gap.

\emph{From qualitative to quantitative.} The unconditional bound
$p(n)/n\to\infty$ of \cite{AB07} was later sharpened by Bugeaud and Evertse
\cite{BE08}, who replaced the Subspace Theorem by its quantitative form and
obtained $p(n,\xi,b)\ge n(\log n)^{0.09}$ for infinitely many $n$. The
corresponding upgrade here is not automatic, and the obstruction can be named.
A quantitative Subspace Theorem bounds the number of \emph{subspaces} carrying
the large solutions, but not the coefficients of those subspaces, and not the
number of solutions in each. In the proof of Theorem~\ref{thm:kernel} each
subspace is disposed of either by a floor $|\beta_1|/2$ or $|\beta_2|$ that
depends on its coefficients, or by an appeal to Lemma~\ref{lem:ridout} on a
fixed gap, and both steps are qualitative.

\begin{problem}\label{prob:quant}
Make the analysis of the exceptional subspaces in Theorem~\ref{thm:kernel}
quantitative --- by controlling the coefficients $\beta_1,\beta_2$ and counting
the solutions of Lemma~\ref{lem:ridout} on each fixed gap --- and deduce a bound
of the shape $\pT(k)\ge k(\log k)^{\delta}$ for infinitely many $k$.
\end{problem}

\emph{Generality.} Only two features of $3/2$ were used: that the orbit points
are built from the finitely many primes $2$ and $3$, and that the numerator of
$(3/2)^{n}$ is odd --- the source of every positivity step
(Lemma~\ref{lem:floor}, the oddness of $N$ in Theorem~\ref{thm:kernel}). For
coprime $p>q\ge2$ the second becomes $N\equiv p^{c}\not\equiv0\pmod q$, and the
proof of Theorem~\ref{thm:kernel} goes through with $S$ the set of primes
dividing $pq$; so does the proof of Lemma~\ref{lem:ridout}, which already
allows a rational multiplier. We
have not written out the remaining details for general $\xi$, nor formalized
any of this.

\begin{problem}\label{prob:pq}
Prove superlinearity of the steering word of $\xi\,(p/q)^{n}$ for every
$\xi\in\Q^{\times}$ and all coprime $p>q\ge2$. The growth ceiling becomes
$k\le(c+1)\log_{q}(p/q)$, so the analogue of Theorem~\ref{thm:t01} carries the
constant $\log q/\log(p/q)$.
\end{problem}

Finally, the Subspace Theorem in more variables suggests \emph{higher-order}
repulsion: finiteness of near-integer coincidences among three or more orbit
points. The analysis of the exceptional subspaces would then no longer be
elementary, because two ratios such as $(3/2)^{c-b}$ and $(3/2)^{b-a}$ can stay
bounded independently of each other. Exploiting such a statement would also
require a combinatorial step the present argument lacks: the pigeonhole of
Theorem~\ref{thm:reduction} extracts coincidences two at a time.

\subsection{Symbolic dynamics}

The sharpest consequence for this field is the dichotomy of
Section~\ref{sec:landscape}: if $\big((3/2)^n\big)$ is dense modulo one then
$T$ is not morphic, so a morphism generating $T$ would refute the density
conjecture. Theorem~\ref{thm:intro-main} is what makes this non-vacuous, having
removed the classes $\Theta(1)$ and $\Theta(k)$ from Pansiot's list; the
remaining candidates are $\Theta(k\log\log k)$, $\Theta(k\log k)$ and
$\Theta(k^{2})$. A systematic search for a morphism generating the observed
prefix of $T$ is therefore a well-posed experiment with an asymmetric payoff:
a hit would settle a hard number-theoretic conjecture in the negative, while
failure accumulates evidence for density.

The asymptotics are best phrased through the topological entropy of the
subshift generated by $T$,
\[
   h(T)\;=\;\lim_{k\to\infty}\frac{\log \pT(k)}{k},
\]
which exists by subadditivity. Our results say nothing about it:
superlinearity is compatible with $h(T)=0$. What is known is the bracket
$h(T)\le\log3$ from Kopra's trace subshift \cite{Kop21}, while
Corollary~\ref{cor:dense} gives $h(T)\ge\log(3/2)$ under density. Contrapositively,
\emph{a proof that the steering subshift has zero entropy would refute the density of
$\big((3/2)^n\big)$ modulo one} --- a purely symbolic-dynamical route into the
problem, and one on which the trivial ceiling $\pT(k)\le5^{k}$ is silent.

\begin{problem}\label{prob:entropy}
Decide whether $h(T)>0$. More modestly, close the gap between the base $3$ of
Kopra's ceiling and the base $3/2$ that density would demand.
\end{problem}

Two structural obstacles are worth recording for anyone who takes this up.
First, Proposition~\ref{prop:packing} is one-way: a length-$k$ factor is
governed not by $\eps_a$ alone but by the pair $(\eps_a,m_a\bmod 2^{k})$, so the
natural phase space is $[-\tfrac12,\tfrac12)\times\Z_2$ rather than the
interval, and it is there that a converse --- factors corresponding to separated
points --- would have to be sought. Establishing one would make the transfer
between complexity and equidistribution two-way, which it currently is not.
Second, the parity word $b$ resists the method for a reason isolated in the
remark after Proposition~\ref{prop:parity}: a minimal desynchronization
$|t_{c+j}-t_{a+j}|=2$ carries no constraint on $|\eps_c-\eps_a|$, so
Lemma~\ref{lem:contraction} has no $b$-analogue. Any unconditional superlinear
bound for $p_b$ needs a substitute --- plausibly a statistical bound on how often
minimal desynchronization can occur along a window.

\section{Acknowledgements}
A previous version of this paper closed the kernel (K) by a longer route: a gap
dichotomy whose infinite branch used arguments from the paper \cite{NKR25} of
Nair, Kumar and Rout. We thank Yann Bugeaud for the simplification of the argument.

The author utilized Claude Code as an AI coding assistant to aid in the Lean 4 formalization of the proofs presented in this paper. The author directed and reviewed all generated code and takes full responsibility for the mathematical integrity and final content of the work.

\appendix
\section{Numerical exploration}\label{app:data}

The steering word is exactly computable, and it repays a look. Everything in
this appendix was produced from the division-free integer recurrence
\[
   m_{n+1}=\frac{3m_n+t_n}{2},\qquad
   R_{n+1}=3R_n-2^{\,n}t_n,\qquad
   \eps_n=\frac{R_n}{2^{\,n}},
\]
in which $t_n$ is the unique integer with $\eps_{n+1}\in[-\tfrac12,\tfrac12)$ and
$t_n\equiv m_n\pmod 2$; no floating-point arithmetic enters the generation of
$T$. We use the first $N=10^{6}$ letters. Write $\pT^{(N)}(k)$ for the number of
distinct length-$k$ factors of the prefix $t_0t_1\cdots t_{N-1}$. Two properties
of this statistic govern how it may be read: it is a \emph{lower} bound for
$\pT(k)$, every factor of a prefix being a factor of $T$; and it can never
exceed $N-k+1$, so it stops being informative as soon as $\pT(k)$ approaches
$N$. Nothing in this appendix belongs to the formal development, and
Conjecture~\ref{conj:exact} is a conjecture, not a theorem.

\subsection*{Letters and pairs}

Suppose $\eps_n$ equidistributes in $[-\tfrac12,\tfrac12)$ and the parity
$m_n\bmod2$ behaves like an independent fair coin. Since $t_n$ is one of the two
integers in $(3\eps_n-1,3\eps_n+1]$, selected between them by the parity
constraint, the model predicts the letter $0$ with probability $\tfrac13$, the
letters $\pm1$ with probability $\tfrac14$ and the letters $\pm2$ with
probability $\tfrac1{12}$. Measured over $10^{6}$ letters:

\begin{center}\footnotesize
\begin{tabular}{lccccc}
\hline
letter & $-2$ & $-1$ & $0$ & $1$ & $2$\\
\hline
measured & $0.08302$ & $0.25017$ & $0.33362$ & $0.25028$ & $0.08292$\\
model    & $1/12$    & $1/4$     & $1/3$     & $1/4$     & $1/12$\\
\hline
\end{tabular}
\end{center}

\noindent
Each cell agrees to within $5\cdot10^{-4}$. The equidistribution the model
assumes is visible directly in the orbit: a $50$-bin histogram of $\eps_n$,
$n<10^{6}$, has occupancies between $19\,732$ and $20\,305$ against an expected
$20\,000$. Pairs of letters, by contrast, are constrained outright.

\begin{proposition}\label{prop:pairsum}
For every $n$, $|t_n+t_{n+1}|\le2$. Consequently the six pairs
$(\pm2,\pm2)$, $(\pm2,\pm1)$, $(\pm1,\pm2)$ with equal signs never occur, and
$\pT(2)=19$.
\end{proposition}

\begin{proof}[Proof]
Applying $t_n=3\eps_n-2\eps_{n+1}$ twice,
\[
   t_n+t_{n+1}=3\eps_n+\eps_{n+1}-2\eps_{n+2}.
\]
Since $-\tfrac12\le\eps_j<\tfrac12$ for every $j$, the right-hand side is
strictly below $\tfrac32+\tfrac12+1=3$ and strictly above
$-\tfrac32-\tfrac12-1=-3$; being an integer, it lies in $[-2,2]$. The excluded
pairs are exactly the six elements of $\{-2,\dots,2\}^{2}$ with
$|t_n+t_{n+1}|\ge3$, and $25-6=19$ is the measured value.
\end{proof}

The rule is not the whole story. The subshift of finite type it defines has
$75$ words of length $3$, whereas $\pT(3)=65$, so constraints of range~$3$ are
present as well; and it yields no better ceiling than
Section~\ref{sec:landscape} already has, its entropy being $\log3.935\ldots$
against Kopra's $\log3$.

\subsection*{The complexity profile}

\begin{figure}[htb]
\includegraphics[width=\linewidth]{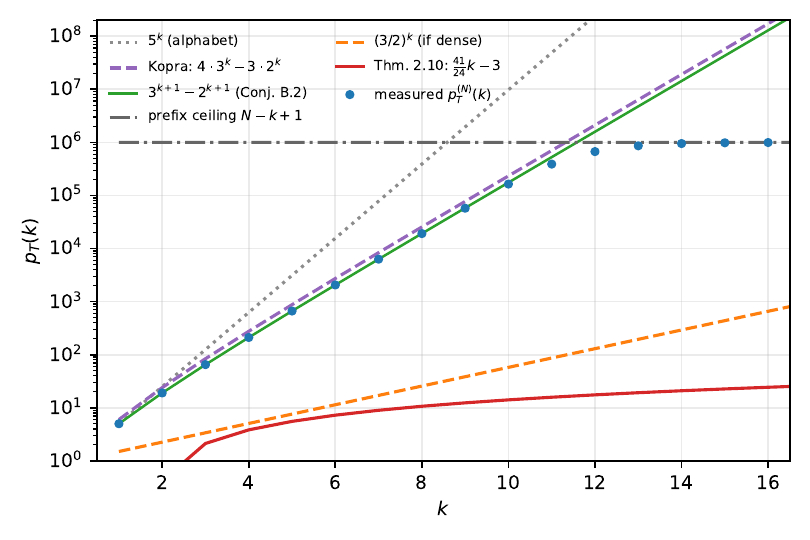}
\caption{Measured $\pT^{(N)}(k)$ at $N=10^{6}$ against the proven and
conjectural environment. The data follow the curve $3^{k+1}-2^{k+1}$ for as long
as $10^{6}$ letters suffice to exhibit every factor, fall away from it as
sampling fails, and finally saturate at the prefix ceiling $N-k+1$; both effects
are artefacts of the finite sample. Note the distance to the certified floor of
Theorem~\ref{thm:t01} and to the bound $(3/2)^k$ that density would force.}
\label{fig:complexity}
\end{figure}

The measured profile is, for $k\le9$,

\begin{center}\footnotesize
\begin{tabular}{lrrrrrrrrr}
\hline
$k$ & 1 & 2 & 3 & 4 & 5 & 6 & 7 & 8 & 9\\
\hline
$\pT^{(N)}(k)$      & 5 & 19 & 65 & 211 & 665 & 2059 & 6304 & 19138 & 57447\\
$3^{k+1}-2^{k+1}$   & 5 & 19 & 65 & 211 & 665 & 2059 & 6305 & 19171 & 58025\\
\hline
\end{tabular}
\end{center}

\begin{conjecture}\label{conj:exact}
$\pT(k)=3^{\,k+1}-2^{\,k+1}$ for every $k\ge1$.
\end{conjecture}

Agreement is exact for $1\le k\le6$. At $k=7$ and $k=8$ the prefix falls short
by $1$ and by $33$, which is the expected effect of rarity rather than evidence
against the formula: the letter frequencies are far from uniform, so a factor
built from extreme letters has probability about $12^{-k}$ per position --- at
$k=7$, some three hundredths of an occurrence in $10^{6}$ positions.
Figure~\ref{fig:ratio} exhibits the effect: for each $N$ the ratio
$\pT^{(N)}(k)/(3^{k+1}-2^{k+1})$ sits at exactly $1$ up to a threshold that
moves right as $N$ grows, and collapses beyond it.

\begin{figure}[htb]
\includegraphics[width=\linewidth]{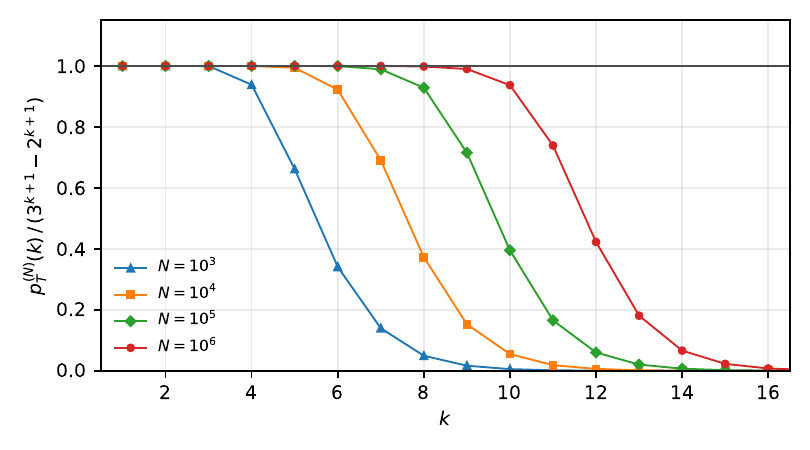}
\caption{$\pT^{(N)}(k)$ normalized by the conjectured value, for four prefix
lengths. The plateau at $1$ lengthens with $N$; the collapse to its right is
under-sampling, not a failure of the formula.}
\label{fig:ratio}
\end{figure}

Three things would follow from Conjecture~\ref{conj:exact}. The topological
entropy of the steering subshift would be exactly $\log3$, so the base of
Kopra's ceiling is sharp and only its constant ($4$ against $3$) is not; this
would settle Problem~\ref{prob:entropy} in both its parts, with
$h(T)=\log3>0$. The complexity would exceed by the exponential factor $2^{k}$ the
bound $\lceil(3/2)^{k}\rceil-1$ that density forces through
Corollary~\ref{cor:dense}; the $(3/2)^{k}$ there counts the $\eps$-coordinate
alone, and the missing $2^{k}$ is exactly the parity coordinate $m_a\bmod2^{k}$,
their product $3^{k}$ being one more reason to regard
$[-\tfrac12,\tfrac12)\times\Z_2$ as the right phase space
(Section~\ref{sec:outlook}). And the distance to what is proved is stark: at
$k=8$ the certified floor of Theorem~\ref{thm:t01} reads $10.7$ against a
measured $19\,138$.

\subsection*{Repetitions}

\begin{figure}[htb]
\includegraphics[width=\linewidth]{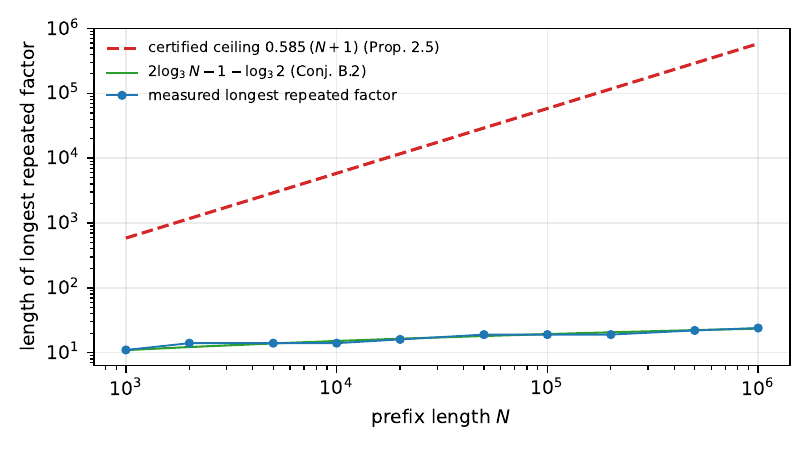}
\caption{Longest repeated factor in a prefix of length $N$, against the ceiling
certified by Proposition~\ref{prop:ceiling} and against the birthday prediction
of Conjecture~\ref{conj:exact}. The certified ceiling is linear in $N$; the
repetitions actually present die out logarithmically.}
\label{fig:repeat}
\end{figure}

Proposition~\ref{prop:ceiling} bounds repetition lengths linearly in the
position: inside a prefix of length $N$ it permits repeated factors of length up
to $0.585\,(N+1)$, that is $585\,000$ at $N=10^{6}$. The longest repeated factor
actually present there has length $24$. What the measurements track instead is
\[
   2\log_3N-\log_32-1,
\]
the length at which, for a word of complexity $3^{k+1}$, the expected number of
coinciding pairs among $N$ windows falls below one; the fit is within one letter
over three decades. This tests Conjecture~\ref{conj:exact} at $k\approx24$, far
outside the range in which $\pT(k)$ can be counted directly --- though it tests
the exponential base only, since replacing $3^{k+1}$ by Kopra's $4\cdot3^{k}$
would move the prediction by $\log_3(4/3)\approx0.26$, well below the resolution
of the data.

The four orders of magnitude between the two curves of
Figure~\ref{fig:repeat} are the visual form of the gap that
Sections~\ref{sec:stage1}--\ref{sec:stage2} do not close. The mechanism of
Section~\ref{sec:stage0} forbids only very long repetitions, whereas in the word
itself repetitions die out logarithmically; an argument that saw the true decay
rate would deliver an exponential complexity bound, and the Diophantine input
used here sees only the linear ceiling.


\begin{thebibliography}{NKR25}

\bibitem[AB07]{AB07}
B.~Adamczewski and Y.~Bugeaud,
\emph{On the complexity of algebraic numbers~I. Expansions in integer bases},
Ann. of Math. (2) \textbf{165} (2007), no.~2, 547--565.

\bibitem[AS03]{AS03}
J.-P.~Allouche and J.~Shallit,
\emph{Automatic sequences: theory, applications, generalizations},
Cambridge University Press, Cambridge, 2003.

\bibitem[BE08]{BE08}
Y.~Bugeaud and J.-H.~Evertse,
\emph{On two notions of complexity of algebraic numbers},
Acta Arith. \textbf{133} (2008), no.~3, 221--250.

\bibitem[BG06]{BG06}
E.~Bombieri and W.~Gubler,
\emph{Heights in Diophantine geometry},
New Mathematical Monographs \textbf{4}, Cambridge University Press, Cambridge,
2006.

\bibitem[Cas03]{Cas03}
J.~Cassaigne,
\emph{Constructing infinite words of intermediate complexity},
Developments in Language Theory (DLT 2002), Lecture Notes in Comput. Sci.
\textbf{2450}, Springer, Berlin, 2003, 173--184.

\bibitem[CN10]{CN10}
J.~Cassaigne and F.~Nicolas,
\emph{Factor complexity},
in: V.~Berth\'e and M.~Rigo (eds.), Combinatorics, automata and number theory,
Encyclopedia Math. Appl. \textbf{135}, Cambridge University Press, Cambridge,
2010, 163--247.

\bibitem[Cob72]{Cob72}
A.~Cobham,
\emph{Uniform tag sequences},
Math. Systems Theory \textbf{6} (1972), 164--192.

\bibitem[CZ04]{CZ04}
P.~Corvaja and U.~Zannier,
\emph{On the rational approximations to the powers of an algebraic number:
solution of two problems of Mahler and Mend\`es France},
Acta Math. \textbf{193} (2004), no.~2, 175--191. arXiv:math/0403522.

\bibitem[Dub09]{Dub09}
A.~Dubickas,
\emph{On integer sequences generated by linear maps},
Glasg. Math. J. \textbf{51} (2009), no.~2, 243--252.

\bibitem[DN05]{DN05}
A.~Dubickas and A.~Novikas,
\emph{Integer parts of powers of rational numbers},
Math. Z. \textbf{251} (2005), no.~3, 635--648.

\bibitem[AFS08]{AFS08}
S.~Akiyama, C.~Frougny, and J.~Sakarovitch,
\emph{Powers of rationals modulo $1$ and rational base number systems},
Israel J. Math. \textbf{168} (2008), 53--91.

\bibitem[DHS99]{DHS99}
F.~Durand, B.~Host, and C.~Skau,
\emph{Substitution dynamical systems, Bratteli diagrams and dimension groups},
Ergodic Theory Dynam. Systems \textbf{19} (1999), no.~4, 953--993.

\bibitem[EM25]{EM25}
H.~Erazo and C.~G.~Moreira,
\emph{The Heinis spectrum has non-empty interior},
Combinatorics on Words (WORDS 2025), Lecture Notes in Comput. Sci., vol 15729. Springer,
Cham, 2025.

\bibitem[Fer99]{Fer99}
S.~Ferenczi,
\emph{Complexity of sequences and dynamical systems},
Discrete Math. \textbf{206} (1999), no.~1--3, 145--154.

\bibitem[FLP95]{FLP95}
L.~Flatto, J.~C.~Lagarias, and A.~D.~Pollington,
\emph{On the range of fractional parts $\{\xi(p/q)^n\}$},
Acta Arith. \textbf{70} (1995), no.~2, 125--147.

\bibitem[HL25]{HL25}
A.~Hilion and G.~Levitt,
\emph{A Pansiot-type subword complexity theorem for automorphisms of free
groups},
Israel J. Math. 270, 59–93 (2025). arXiv:2208.00676.

\bibitem[Kop21]{Kop21}
J.~Kopra,
\emph{On the trace subshifts of fractional multiplication automata},
Theoret. Comput. Sci. \textbf{851} (2021), 92--110.

\bibitem[Mah57]{Mah57}
K.~Mahler,
\emph{On the fractional parts of the powers of a rational number~II},
Mathematika \textbf{4} (1957), 122--124.

\bibitem[MH38]{MH38}
M.~Morse and G.~A.~Hedlund,
\emph{Symbolic dynamics},
Amer. J. Math. \textbf{60} (1938), no.~4, 815--866.

\bibitem[Pan84]{Pan84}
J.-J.~Pansiot,
\emph{Complexit\'e des facteurs des mots infinis engendr\'es par morphismes
it\'er\'es},
Automata, Languages and Programming (ICALP 1984), Lecture Notes in Comput. Sci.
\textbf{172}, Springer, Berlin, 1984, 380--389.

\bibitem[Puz25]{Puz25}
S.~Puzynina,
\emph{Complexity of infinite words},
Machines, Computations, and Universality (MCU 2024), Lecture Notes in Comput.
Sci. \textbf{15270}, Springer, Cham, 2025, 1--16.

\bibitem[NKR25]{NKR25}
P.~S.~Nair, V.~Kumar, and S.~S.~Rout,
\emph{Algebraic approximations to linear combinations of $S$-units},
arXiv:2506.02898 (v3, 18 Nov 2025).

\bibitem[Rid57]{Rid57}
D.~Ridout,
\emph{Rational approximations to algebraic numbers},
Mathematika \textbf{4} (1957), 125--131.

\bibitem[Sch77]{Sch77}
H.~P.~Schlickewei,
\emph{The $\mathfrak{p}$-adic Thue--Siegel--Roth--Schmidt theorem},
Arch. Math. (Basel) \textbf{29} (1977), no.~3, 267--270.

\bibitem[Sch91]{Sch91}
Schmidt, Wolfgang M.,
\emph{Diophantine approximations and Diophantine equations.}
Springer, 2006.

\bibitem[Zud07]{Zud07}
W.~Zudilin,
\emph{A new lower bound for $\lVert(3/2)^k\rVert$},
J. Th\'eor. Nombres Bordeaux \textbf{19} (2007), no.~1, 311--323.

\end{thebibliography}
\end{document}